\documentclass[11pt]{amsart}
\usepackage[english]{babel}
\usepackage[utf8]{inputenc}
\usepackage{amsmath}
\usepackage{graphicx}
\usepackage{amssymb}
\usepackage{amsthm}
\usepackage{tikz-cd}
\usetikzlibrary{calc}
\usepackage{mathrsfs}
\usepackage[colorinlistoftodos]{todonotes}
\usepackage{enumitem}
\usepackage{yfonts}
\usepackage{xcolor}
\usepackage{mathtools}
\usepackage{hyperref}
\usepackage{comment}
\newcommand{\Isom}{\mathrm{Isom}}

\newcommand{\F}{\mathbb{F}}
\newcommand{\N}{\mathbb{N}}
\newcommand{\Vol}{\operatorname{Vol}}

\DeclareMathOperator{\Fix}{Fix}
\DeclareMathOperator{\Min}{Min}

\title[Betti Numbers of Negatively Curved Orbifolds]{Betti Numbers of Negatively Curved Orbifolds with Coefficients in Arbitrary Fields}
\author{Guy Kapon, Raz Slutsky}

\date{}
\newtheorem{thm}{Theorem}[section]
\newtheorem{lem}[thm]{Lemma}
\newtheorem{prop}[thm]{Proposition}

\newtheorem{claim}[thm]{Claim}

\newcommand{\FF}{\mathbb{F}}

\newcommand{\R}{\mathbb{R}}

\newcommand{\Z}{\mathbb{Z}}

\begin{document}

\begin{abstract}
We show that the Betti numbers of finite-volume negatively curved orbifolds grow at most linearly with the volume, with coefficients in an arbitrary field. In particular, this gives a  linear bound for the Betti numbers of finite-volume hyperbolic orbifolds over $\F_p$. This extends a theorem of Gromov from manifolds to orbifolds in negative curvature, and answers a question of Samet, by strengthening his theorem from characteristic $0$ to arbitrary characteristic. The key new input is a quantitative bound on the homology of spherical quotients.
\end{abstract}

\maketitle

\section{Introduction}

Let $X$ be a Hadamard manifold, i.e., a connected, simply connected, complete Riemannian manifold of non-positive curvature normalized such that $-1 \leq K \leq 0$, and let $\Gamma$
be a discrete subgroup of $\textrm{Isom}(X)$. If $\Gamma$ is torsion-free then $M = X / \Gamma$ is a Riemannian
manifold. Otherwise, $M$ has the structure of an orbifold.

An important special case is that of symmetric spaces, that is,  $X$ is given by $X= G/K$ for $G$ a semi-simple Lie group and $K$ a maximal compact subgroup, and $\Gamma$ is a lattice in $G$, i.e., a discrete subgroup such that $G/\Gamma$ carries a finite $G$-invariant measure. This includes, for example, the important class of hyperbolic manifolds and orbifolds.

In many cases, the volume of a non-positively curved manifold controls the complexity of its topology. One of the first results demonstrating this phenomenon is the following celebrated theorem of Gromov
 from 1985, \cite{BGS85}: 

\begin{thm}[Gromov, 1985]\label{thm:Gromov}
There exists a constant $C_n$, depending only on the dimension, such that for every Hadamard manifold $X$ and torsion-free lattice $\Gamma$ 
$$\sum_{i=1}^n b_i(X / \Gamma) \leq C_n \operatorname{Vol}(X / \Gamma) $$
where $b_i(X / \Gamma)$ are the betti numbers with respect to coefficients in any field.
\end{thm}

In recent years, a range of results connecting lattices' complexity and their co-volume have been proved. See for example the work of Gelander relating the co-volume to the homotopy type of the quotient in \cite{gelander2004homotopy}, works relating the co-volume to the minimal number of generators such as \cite{gelander2020minimal, lubotzky2022asymptotic}, and growth of other homological invariants, for example, \cite{abert2022homology, frkaczyk2022homotopy, CHM, CAM, 7s, abert2017rank, BL12, BL19, BGS20, abert2023Betti}, and versions for general hyperbolic groups in \cite{lazarovich2021volume, lazarovich2023finite}.
Note that most of these results are proved for torsion-free lattices. However, in many settings, and particularly from the perspective of arithmetic and number theory, orbifolds arise very naturally, and many familiar lattices contain torsion. For instance, $\mathrm{SL}_2(\mathbb{Z})$, where the quotient is the modular orbifold. More generally, for a number field $k$ with ring of integers $\mathcal{O}_k$, the lattices $\mathrm{SL}_2(\mathcal{O}_k)$(and their higher-rank analogues) typically have torsion, so the associated locally symmetric quotients are naturally orbifolds. \\

In this paper, we extend Gromov's volume-linear bound to finite-volume orbifolds of negative curvature, with coefficients in an arbitrary field. Namely, we prove:

\begin{thm}\label{thm:intro main theorem}
There exists a constant $C_n$, depending only on the dimension, such that for every Hadamard manifold $X$ 
with $-1 \leq K < 0$ and any lattice $\Gamma$ 
$$\sum_{i=1}^n b_i(X / \Gamma, \mathbb{F}) \leq C_n \operatorname{Vol}(X / \Gamma) $$
for any field $\mathbb{F}$.
\end{thm}

The point is that passing from manifolds to orbifolds introduces finite stabilizers. Over characteristic $0$ these stabilizers are often invisible at the level of Betti numbers, but in positive characteristic they can contribute new homology classes, so one needs uniform control of the topology coming from the singularities.

In recent years, due to connections to number theory, mod-$p$ betti numbers and torsion in homology with $\F_p$-coefficients have attracted a significant amount of attention. In characteristic zero, tools such as L\"{u}ck's approximation theorem \cite{Luck1994} and the work of \cite{7s} provide strong control on normalized Betti numbers along towers of covers in many settings. In positive characteristic, our understanding of homology growth is still much more limited, and the behavior can differ substantially from the characteristic zero picture; for instance, constructions of Avramidi--Okun--Schreve \cite{AvramidiOkunSchreve2021} highlight phenomena for mod-$p$ and torsion homology growth in nonpositive curvature that have no direct characteristic-zero analogue. See the detailed introduction of \cite{abert2022homology} for an overview of this circle of problems, and for the current state-of-the-art in the growth of $\F_p$ homology for chains of higher-rank, torsion-free lattices.\\

For general sequences of non-commensurable lattices, results on the first homology over positive characteristic include \cite{fraczyk2022growth}, and most recently the breakthrough of \cite{fraczyk2023poisson}, showing, among other things, that in higher rank the minimal number of generators grows sub-linearly in the volume. Again, these results are for torsion-free lattices. For example, bounds on the number of generators of lattices with torsion are only known for non-uniform lattices, due to \cite{lubotzky2022asymptotic}.
We also note that the recent work of the second author and Gelander \cite{gelander2025quantitative} shows, by proving a quantitative version of Selberg's lemma, that every arithmetic locally symmetric orbifold is covered by a manifold where the index of the covering is poly-logarithmic in the volume of the orbifold.
However, this does not yield good enough control of the betti numbers for this purpose. Moreover, there are many non-arithmetic locally symmetric orbifolds in rank 1, see for example \cite{gelander2014counting}.\\

In \cite{samet2013betti}, Samet showed that the same conclusion as in Theorem \ref{thm:Gromov} holds for orbifolds, assuming that the curvature is negative and the betti numbers are taken with respect to rings over characteristic 0. Such types of results for orbifolds are usually hard to prove by geometric methods, as a good understanding of the singularities is needed. Theorem~\ref{thm:intro main theorem} extends Samet's result to coefficients in an arbitrary field and answers a question posed in the same paper.

Note also the result of \cite{Senska2023}, where related volume-linear bounds for Betti numbers over arbitrary fields (and, moreover, bounds on logarithmic torsion in homology) are obtained for hyperbolic orbifolds, with constants that depend additionally on a bound on the orders of finite subgroups of $\Gamma$. However, for co-compact lattices, the maximal order of a finite subgroup can grow with the co-volume; see \cite{gelander2025quantitative} for examples.


The main difference between the torsion-free case and the general case is that the finite subgroups of $\Gamma$ can introduce additional homology in positive characteristic. Controlling these finite subgroups and, in particular, their linear actions on the tangent space is therefore a central ingredient in our proof.

Our main topological input is the following uniform bound on the homology of spherical quotients.

\begin{thm}\label{thm:intro main topological theorem}
Let $k \in \N$, and let $G$ be any finite group acting linearly on $\mathbb{S}^k$, then there exists a constant $C$, depending only on $k$, such that $$b_i(\mathbb{S}^k/G, \mathbb{F}) \leq C $$ for every field $\mathbb{F}$. In fact, one can take $C = 3^k(k+1)!^{\log(3k)}$.
\end{thm}

Once we prove this theorem, we use the methods developed by Gromov, Gelander and Samet to analyse the thick and thin parts of the quotient and bound their topological complexity in terms of the volume.\\


We will indicate where the assumption of negative curvature is used. It is quite likely that the result can be extended to orbifolds of non-positive curvature, hence completing all remaining cases of \ref{thm:Gromov}. 

\subsection{Finite Group Actions on Spheres}
It is natural to ask whether one can remove the \emph{linearity} hypothesis in
Theorem~\ref{thm:intro main topological theorem} by some general ``linearization''
or approximation principle for finite group actions on spheres.

In the purely topological category, the answer is generally no: there are well-known ``wild'' finite group actions on spheres which are not even locally linear. In dimension three, classic constructions of Bing~\cite{bing1952} and Montgomery--Zippin~\cite{montgomery1954examples} exhibit involutions on $S^3$ with pathological fixed sets (e.g.\ the Alexander horned sphere or a Cantor set). In higher dimensions, there are also topological actions by finite groups which do not embed into $\mathrm{O}(n+1)$; see, for instance, Zimmermann~\cite{zimmermann2017topological}.

In the smooth category the picture is more rigid in low dimensions. In dimension $3$, smooth finite group actions are governed by geometrization (and are smoothly conjugate to isometric actions); see e.g.\ \cite{dinkelbach2009equivariant} and the discussion around the Smith conjecture in \cite{morganbass1984}. In dimension $4$, there are strong linearization results in many cases; see \cite{ChenKwasikSchultz2016}. In higher dimensions, however, smooth finite group actions need not be conjugate to linear ones, and the distinction between linear and nonlinear actions becomes subtler.

Even for free actions, nonlinearity already appears for cyclic groups: there exist smooth free actions of $\Z/m$ on $S^{2k-1}$ which are not smoothly conjugate to any action arising from a unitary representation. The resulting quotients are the classical \emph{fake lens spaces}; see Browder--Petrie--Wall~\cite{BPW71} and the subsequent classification results surveyed and developed for example in Macko--Wegner~\cite{MackoWegner11}.

Moreover, there are smooth free actions on homotopy spheres whose orbit spaces are not even homotopy equivalent to any linear spherical quotient. Indeed, there exist free smooth actions of odd-order non-abelian metacyclic groups
\[
\pi \;=\; \Z/p \rtimes \Z/q
\]
on homotopy spheres, see Petrie~\cite{Petrie71} and the discussion in Hambleton's survey~\cite{Hambleton14}. On the other hand, Zassenhaus showed that any finite group admitting a free orthogonal action on a sphere must satisfy the \emph{$pq$-conditions} (every subgroup of order $pq$, for primes $p,q$, is cyclic), and conversely, solvable groups satisfying the $pq$-conditions admit a free orthogonal representation \cite{Hambleton14}. The above metacyclic groups fail the $pq$-conditions, hence admit no free linear action on any sphere. Therefore, if $M=S^{n}/\pi$ is such a quotient and $M\simeq S^{n}/H$ for some linear $H<\mathrm{O}(n+1)$, then $\pi_1(M)\cong \pi_1(S^{n}/H)\cong H$ would force $H\cong \pi$, contradicting the nonexistence of a free linear action by $\pi$.

\section*{Acknowledgments}
The authors would like to thank Shaul Ragimov for useful conversations. The second author wishes to thank Tsachik Gelander for his continued support and encouragement to work on this problem.

\section{Linear actions of finite groups on spheres}

Let $G\le SO(n)$ be a finite group acting linearly on the unit sphere $S^{n-1}\subset\mathbb{R}^n$.
Since the action is smooth, Illman's equivariant triangulation theorem yields a finite $G$--invariant triangulation of
$S^{n-1}$ \cite{Illman1978}. After barycentric subdivision, we may assume the induced simplicial $G$--action is
\emph{admissible} in the sense that if $g\in G$ preserves a simplex setwise then it fixes it pointwise
\cite[\S1]{PutmanSmith}. In particular, $S^{n-1}$ is a finite admissible $G$--CW complex; for every subgroup
$H\le G$ the fixed point set $(S^{n-1})^H$ is a subcomplex and the quotient $S^{n-1}/H$ inherits a finite CW structure
\cite[\S1]{AdemDavis}.

\subsection{Jordan reduction to a bounded extension of an abelian group}

By Jordan's theorem \cite{jordan1878memoire}, with explicit bounds due to Collins \cite{collins2007jordan}, there exists
a constant $J(n)$ depending only on $n$ such that every finite subgroup of $GL_n(\mathbb{C})$ contains a \emph{normal
abelian} subgroup of index at most $J(n)$. Applying this to $G\le SO(n)\subset GL_n(\mathbb{C})$, choose
\[
A\triangleleft G \ \ \text{abelian},\qquad [G:A]\le J(n),
\]
and set $Q:=G/A$, so $|Q|\le J(n)$. Let
\[
X:=S^{n-1}/A.
\]
Normality of $A$ implies that $Q$ acts on $X$, and
\[
S^{n-1}/G \ \cong\ X/Q.
\]
Thus, we will first give a bound for abelian actions (where the size of the group can be arbitrarily large), and then deal with the quotient by $Q$.

\subsection{A uniform bound in the abelian case}

We record an explicit uniform bound (depending only on $n$) for quotients by finite abelian subgroups of $SO(n)$.

\begin{thm}\label{thm:abelian_SO_betti_bound}
Let $A$ be a finite abelian subgroup of $SO(n)$ acting linearly on $S^{n-1}$, and let $\FF$ be any field. Then
\[
\sum_{i=0}^{n-1}\dim_{\FF}H^i(S^{n-1}/A;\FF)\ \le\ 3^n.
\]
\end{thm}

\begin{proof}
Set $X:=S^{n-1}/A$. Since $A\subset O(n)$ is finite abelian, its elements are commuting orthogonal operators on
$\mathbb{R}^n$, hence $\mathbb{R}^n$ decomposes as an orthogonal direct sum of $A$--invariant real irreducible
subrepresentations. Over $\mathbb{R}$ these are either $1$--dimensional sign representations or $2$--dimensional
rotation blocks. Thus, there is an orthogonal decomposition
\[
\mathbb{R}^n \cong \Bigl(\bigoplus_{j=1}^r V_j\Bigr)\oplus \Bigl(\bigoplus_{k=1}^s W_k\Bigr),
\qquad \dim V_j=2,\ \dim W_k=1,
\]
such that $A$ acts on each $V_j$ by rotations (via a character $\chi_j:A\to S^1$ after identifying $V_j\cong\mathbb{C}$)
and on each $W_k$ by multiplication by $\pm1$ (via a character $\varepsilon_k:A\to\{\pm1\}$). In particular $2r+s=n$.
Let $N:=r+s$ and index the blocks by $B_\alpha$ ($\alpha=1,\dots,N$), where $B_\alpha=V_\alpha$ for $\alpha\le r$ and
$B_{r+k}=W_k$ for $1\le k\le s$. Then $N\le n$, and every $v\in\mathbb{R}^n$ decomposes uniquely as
$v=(v^{(1)},\dots,v^{(N)})$ with $v^{(\alpha)}\in B_\alpha$.

Define a continuous map $\mu:S^{n-1}\to\Delta^{N-1}$ by
\[
\mu(v)=\bigl(\|v^{(1)}\|^2,\dots,\|v^{(N)}\|^2\bigr),
\]
where $\Delta^{N-1}=\{(u_1,\dots,u_N)\in\mathbb{R}^N:\ u_i\ge0,\ \sum_i u_i=1\}$ is the standard simplex.
Because $A$ acts orthogonally on each block, it preserves $\|v^{(\alpha)}\|$, hence $\mu$ is $A$--invariant and descends
to a continuous map $\bar\mu:X\to\Delta^{N-1}$.

For each $\alpha\in\{1,\dots,N\}$ let $\mathcal{V}_\alpha:=\{u\in\Delta^{N-1}:\ u_\alpha>0\}$, so that
$\{\mathcal{V}_\alpha\}$ is an open cover of $\Delta^{N-1}$, and set $U_\alpha:=\bar\mu^{-1}(\mathcal{V}_\alpha)$, an
open cover of $X$. For a nonempty $J\subseteq\{1,\dots,N\}$ write
\[
\mathcal{V}_J:=\bigcap_{\alpha\in J}\mathcal{V}_\alpha,
\qquad
U_J:=\bigcap_{\alpha\in J}U_\alpha=\bar\mu^{-1}(\mathcal{V}_J).
\]
Upstairs, let $W_J:=\mu^{-1}(\mathcal{V}_J)\subset S^{n-1}$, so that $U_J=W_J/A$. Let
\[
E_J:=\bigoplus_{\alpha\in J}B_\alpha\subset\mathbb{R}^n,
\qquad
S_J:=S^{n-1}\cap E_J.
\]

\begin{lem}\label{lem:WJ_retract}
The inclusion $W_J\cap S_J\hookrightarrow W_J$ is an $A$--equivariant deformation retraction.
\end{lem}

\begin{proof}
Write $v=(v_J,v_{J^c})$ with $v_J\in E_J$ and $v_{J^c}\in E_{J^c}$. If $v\in W_J$, then each block in $J$ is nonzero,
hence $v_J\neq 0$. Define
\[
H_t(v):=\frac{(v_J,(1-t)v_{J^c})}{\|(v_J,(1-t)v_{J^c})\|}\qquad (t\in[0,1]).
\]
The denominator never vanishes since $v_J\neq 0$, and scaling $v_{J^c}$ does not affect the nonvanishing of blocks in
$J$, so $H_t(v)\in W_J$ for all $t$. Moreover $H_1(v)\in W_J\cap S_J$. The construction uses only orthogonal
projections and norms, hence is $A$--equivariant.
\end{proof}

By Lemma~\ref{lem:WJ_retract}, passing to quotients gives $U_J=W_J/A\simeq (W_J\cap S_J)/A$.

\begin{lem}\label{lem:WJ_model}
Let $a=a(J)$ be the number of $2$--dimensional blocks among $\{B_\alpha\}_{\alpha\in J}$ and $b=b(J)$ the number of
$1$--dimensional blocks among them. Then there is an $A$--equivariant homeomorphism
\[
W_J\cap S_J \cong (S^1)^a \times (S^0)^b \times \Delta^{a+b-1}_{>0},
\]
where $\Delta^{a+b-1}_{>0}=\{(t_1,\dots,t_{a+b})\in\mathbb{R}^{a+b}:\ t_i>0,\ \sum_i t_i=1\}$ is the open simplex.
The $A$--action is by translations on $(S^1)^a$ and $(S^0)^b$, and is trivial on $\Delta^{a+b-1}_{>0}$.
\end{lem}

\begin{proof}
On each complex block write $z=re^{i\theta}$ with $r>0$ and $\theta\in S^1$, and on each real block write
$x=\sigma s$ with $\sigma\in S^0=\{\pm1\}$ and $s>0$. The defining conditions for $W_J\cap S_J$ are precisely that all
$r,s$ are positive and that $\sum r^2+\sum s^2=1$. Thus $(r_1^2,\dots,r_a^2,s_1^2,\dots,s_b^2)$ determines a point of
$\Delta^{a+b-1}_{>0}$, while the angles and signs determine a point of $(S^1)^a\times(S^0)^b$, giving the stated
homeomorphism. The $A$--action rotates each $S^1$ factor and flips each $S^0$ factor via the corresponding characters,
while preserving radii, so the simplex factor is fixed.
\end{proof}

Since $\Delta^{a+b-1}_{>0}$ is contractible and $A$ acts trivially on it, Lemma~\ref{lem:WJ_model} yields
\[
U_J \simeq \bigl((S^1)^a\times (S^0)^b\bigr)/A.
\]

\begin{lem}\label{lem:torus_components}
The space $\bigl((S^1)^a\times (S^0)^b\bigr)/A$ is a disjoint union of at most $2^b$ components, each homotopy
equivalent to a torus of dimension $a$.
\end{lem}

\begin{proof}
Let $D:=(S^0)^b$, a discrete set of size $2^b$. The $A$--action on $(S^1)^a\times D$ has the form
$g\cdot(t,d)=(t+\phi(g),\, d+\psi(g))$ for homomorphisms $\phi:A\to(S^1)^a$ and $\psi:A\to D$. Hence the quotient
decomposes into connected components indexed by $D/A$, so the number of components is at most $|D|=2^b$.
Fix an orbit $\mathcal{O}\subset D$ and $d\in\mathcal{O}$. The stabilizer $A_d$ acts on $(S^1)^a$ by translations
through the finite subgroup $\phi(A_d)\subset (S^1)^a$, and $(S^1)^a/\phi(A_d)$ is again a torus of dimension $a$.
\end{proof}

Lemma~\ref{lem:torus_components} implies that over any field $\FF$, each component contributes total Betti number $2^a$,
hence
\[
\sum_q \dim_{\FF} H^q(U_J;\FF) \le 2^b\cdot 2^a = 2^{a+b}=2^{|J|}.
\]

Now apply the \v{C}ech (equivalently, Mayer--Vietoris) spectral sequence for the finite open cover
$\mathcal{U}=\{U_\alpha\}_{\alpha=1}^N$ with coefficients in $\FF$:
\[
E_1^{p,q}=\bigoplus_{\substack{J\subseteq\{1,\dots,N\}\\ |J|=p+1}} H^q(U_J;\FF)
\ \Longrightarrow\ H^{p+q}(X;\FF).
\]
Taking dimensions over a field and using that differentials can only decrease total dimension yields
\[
\begin{aligned}
\sum_i \dim_{\FF} H^i(X;\FF)
&\le \sum_{p,q}\dim_{\FF}E_1^{p,q} \\
&= \sum_{\emptyset\neq J\subseteq\{1,\dots,N\}} \ \sum_q \dim_{\FF} H^q(U_J;\FF) \\
&\le \sum_{\emptyset\neq J\subseteq\{1,\dots,N\}} 2^{|J|}.
\end{aligned}
\]
Finally,
\[
\sum_{\emptyset\neq J\subseteq\{1,\dots,N\}} 2^{|J|}
= \sum_{k=1}^N \binom{N}{k}2^k
= (1+2)^N-1
= 3^N-1
\le 3^N
\le 3^n,
\]
since $N\le n$.
\end{proof}

\subsection{A characteristic-$p$ bound for finite CW complexes}\label{subsec:orbit-space-bound}

Throughout this subsection, $\FF$ denotes a field of characteristic $p$ and we write
\[
b_i(Y):=\dim_{\FF}H_i(Y;\FF).
\]
We denote by $C_p$ the cyclic group of order $p$. We will also interchange homology and cohomology since over a field one has $\dim_{\FF}H_i(Y;\FF)=\dim_{\FF}H^i(Y;\FF)$ by universal coefficients.

\subsubsection{The cyclic case}

\begin{thm}\label{thm:orbit_cyclic}
Let $C_p$ act admissibly on a finite $d$--dimensional CW complex $Y$. Assume $b_i(Y)\le k$ for all $i$.
Then for all $n \geq 0$ one has
\[
b_n(Y/C_p)\ \le\ 3(d+1)k
\]
\end{thm}

\begin{proof}
Since $Y/C_p$ is $d$--dimensional, we have $H_n(Y/C_p;\FF)=0$ for $n>d$.

By the Smith--Floyd inequality for $p$--groups (over $\mathbb{F}_p$) \cite[Theorem B(i)]{PutmanSmith},
\begin{equation}\label{eq:cyclic-smith-floyd}
\sum_{i\ge 0}\dim_{\mathbb{F}_p}H_i(Y^{C_p};\mathbb{F}_p)
\ \le\
\sum_{i\ge 0}\dim_{\mathbb{F}_p}H_i(Y;\mathbb{F}_p).
\end{equation}
Since $\FF$ is flat over $\mathbb{F}_p$, extension of scalars gives
$H_i(Z;\FF)\cong H_i(Z;\mathbb{F}_p)\otimes_{\mathbb{F}_p}\FF$ for any CW complex $Z$
(cf.\ \cite[Thm.~3.6.1]{Weibel}), hence
$\dim_{\FF}H_i(Z;\FF)=\dim_{\mathbb{F}_p}H_i(Z;\mathbb{F}_p)$.
Therefore,
\[
\sum_i b_i(Y^{C_p})\ \le\ \sum_i b_i(Y)\ \le\ (d+1)k.
\]
In particular, for every $i\ge 0$,
\[
b_i(Y^{C_p})\ \le\ (d+1)k.
\]

From the long exact sequence of the pair $(Y,Y^{C_p})$ \cite[Thm.~2.16]{HatcherAT}, for $t\ge 0$,
\[
b_t(Y,Y^{C_p})
\ \le\ b_t(Y)+b_{t-1}(Y^{C_p})
\ \le\ k+(d+1)k
\ =\ (d+2)k,
\]
with $b_{-1}(Y^{C_p})=0$. Summing over $t=0,\dots,d$ gives
\[
\sum_{t=0}^d b_t(Y,Y^{C_p})
\le \sum_{t=0}^d b_t(Y)+\sum_{t=0}^{d-1} b_t(Y^{C_p})
\le (d+1)k+(d+1)k
=2(d+1)k.
\]

Since every cell of $Y\setminus Y^{C_p}$ has trivial stabilizer, each $C_q(Y,Y^{C_p};\FF)$ is a free
$\FF[C_p]$--module (compare \cite[(1.3)]{BrownLectures}), and taking coinvariants gives a chain isomorphism
\[
\FF\otimes_{\FF[C_p]} C_*(Y,Y^{C_p};\FF)\ \cong\ C_*(Y/C_p,Y^{C_p};\FF)
\]
(cf.\ \cite[Example~5.2(c)]{BrownLectures}). Applying the Hypertor spectral sequence
\cite[\S5.7, Prop.~5.7.6 and App.~5.7.8]{Weibel} yields the (relative) Cartan--Leray spectral sequence
\cite[(5.5)]{BrownLectures}:
\begin{equation}\label{eq:cyclic-cartan-leray}
E^2_{a,b}=H_a\!\bigl(C_p;\,H_b(Y,Y^{C_p};\FF)\bigr)\ \Longrightarrow\ H_{a+b}(Y/C_p,Y^{C_p};\FF).
\end{equation}

For any finite-dimensional $\FF[C_p]$--module $M$, the standard $2$--periodic resolution shows $H_a(C_p;M)$ is a
subquotient of $M$ (see \cite[Example~1.3, (1.9) and Example~3.2]{BrownLectures}), hence
$\dim_{\FF}H_a(C_p;M)\le \dim_{\FF}M$ for all $a\ge 0$.
Consequently, for $0\le n\le d$,
\begin{align*}
b_n(Y/C_p,Y^{C_p})
&\le \sum_{a=0}^n \dim_{\FF}E^2_{a,n-a}
 \le \sum_{a=0}^n b_{n-a}(Y,Y^{C_p})
 = \sum_{t=0}^n b_t(Y,Y^{C_p}) \\
&\le \sum_{t=0}^d b_t(Y,Y^{C_p})
 \le 2(d+1)k.
\end{align*}

Finally, the long exact sequence of the pair $(Y/C_p,Y^{C_p})$ \cite[Thm.~2.16]{HatcherAT} contains the segment
\[
H_n(Y^{C_p};\FF)\longrightarrow H_n(Y/C_p;\FF)\longrightarrow H_n(Y/C_p,Y^{C_p};\FF),
\]
so
\begin{align*}
b_n(Y/C_p)
&\le b_n(Y^{C_p})+b_n(Y/C_p,Y^{C_p}) \\
&\le (d+1)k+2(d+1)k
=3(d+1)k,
\end{align*}
as required.
\end{proof}

\subsubsection{From $C_p$ to $p$--groups}

\begin{lem}\label{lem:pgroup_series}
Let $P$ be a finite $p$--group of order $|P|=p^r$. Then there exists a normal series
\[
1=P_0\ \triangleleft\ P_1\ \triangleleft\ \cdots\ \triangleleft\ P_r=P
\]
such that $P_i/P_{i-1}\cong C_p$ for each $i$.
\end{lem}

\begin{proof}
Since $Z(P)$ is nontrivial, it contains an element of order $p$, hence a subgroup $P_1\le Z(P)$ of order $p$.
Then $P_1\triangleleft P$ and $|P/P_1|=p^{r-1}$. Apply induction to $P/P_1$ and take preimages in $P$.
\end{proof}

\begin{thm}\label{thm:orbit_pgroup}
Let $P$ be a finite $p$--group of order $|P|=p^r$ acting admissibly on a finite $d$--dimensional CW complex $Y$.
Assume $b_i(Y)\le k$ for all $i$. Then for all $n$,
\[
b_n(Y/P)\ \le\ \bigl(3(d+1)\bigr)^r\,k.
\]
\end{thm}

\begin{proof}
Let $1=P_0\triangleleft\cdots\triangleleft P_r=P$ be as in Lemma~\ref{lem:pgroup_series} and set $Y_i:=Y/P_i$.
Since $P_{i-1}\triangleleft P_i$, the quotient $Q_i:=P_i/P_{i-1}\cong C_p$ acts on $Y_{i-1}$ and
$Y_i\cong Y_{i-1}/Q_i$. Set $K_i:=\max_{0\le n\le d} b_n(Y_i)$. By Theorem~\ref{thm:orbit_cyclic} applied to the
$C_p$--action of $Q_i$ on $Y_{i-1}$ (with $k$ replaced by $K_{i-1}$), we have
$K_i\le 3(d+1)\,K_{i-1}$. Iterating gives $K_r\le (3(d+1))^r K_0\le (3(d+1))^r k$.
\end{proof}

\subsubsection{General finite groups via transfer}

\begin{lem}\label{lem:orbit_transfer}
Let $H$ be a finite group acting admissibly on a finite CW complex $Y$, and let $P\in\mathrm{Syl}_p(H)$.
There is a transfer map $H_*(Y/H;\FF)\to H_*(Y/P;\FF)$ whose composition with the map induced by $Y/P\to Y/H$ is
multiplication by $[H:P]$. In particular, $H_*(Y/H;\FF)\hookrightarrow H_*(Y/P;\FF)$ is injective.
\end{lem}

\begin{proof}
Adem--Davis construct a chain-level transfer $C_*(Y/H)\to C_*(Y/P)$ by averaging over coset representatives and show
that composing with the projection $C_*(Y/P)\to C_*(Y/H)$ induces multiplication by $[H:P]$ on homology
\cite[\S2.7]{AdemDavis}. Since $p\nmid [H:P]$, the scalar $[H:P]$ is invertible in $\FF$, so the transfer is injective.
\end{proof}

\begin{thm}\label{thm:orbit_finite_bound}
Let $\FF$ be a field of characteristic $p$. Let $Y$ be a finite $d$--dimensional CW complex equipped with an admissible
cellular action of a finite group $H$. Assume $b_i(Y)\le k$ for all $i$. Then $b_n(Y/H)=0$ for $n>d$, and for all $n$,
\[
b_n(Y/H)\ \le\ \bigl(3(d+1)\bigr)^{\,\lfloor\log_p(|H|)\rfloor}\,k
\ \le\ |H|^{\log_p(3(d+1))}\,k.
\]
\end{thm}

\begin{proof}
Let $P\in\mathrm{Syl}_p(H)$, so $|P|=p^r$ with $r\le \lfloor\log_p(|H|)\rfloor$. By Lemma~\ref{lem:orbit_transfer},
$b_n(Y/H)\le b_n(Y/P)$ for all $n$. Applying Theorem~\ref{thm:orbit_pgroup} to $P$ gives
\[
b_n(Y/H)\ \le\ b_n(Y/P)\ \le\ \bigl(3(d+1)\bigr)^r\,k
\ \le\ \bigl(3(d+1)\bigr)^{\lfloor\log_p(|H|)\rfloor}\,k.
\]
The final inequality follows from $\lfloor\log_p(|H|)\rfloor\le \log_p(|H|)$.
\end{proof}

\subsection{Combining with the Jordan reduction}\label{subsec:jordan-completion}

\begin{thm}\label{thm:sphere_general_bound}
Let $G\le SO(n)$ be finite, acting linearly on $S^{n-1}$. Let $\FF$ be any field and set
$b_i(Y;\FF):=\dim_{\FF}H_i(Y;\FF)=\dim_{\FF}H^i(Y;\FF)$. Choose $A\triangleleft G$ abelian with $[G:A]\le J(n)$ and set
$Q:=G/A$. Then
\[
\sum_{i=0}^{n-1} b_i(S^{n-1}/G;\FF)\ \le\
n\cdot 3^n \cdot |Q|^{\log_2(3n)}
\ \le\
n\cdot 3^n \cdot J(n)^{\log_2(3n)}.
\]

\end{thm}

\begin{proof}
Let $X:=S^{n-1}/A$, so $S^{n-1}/G\cong X/Q$ and $X$ is a finite $(n-1)$--dimensional CW complex. By
Theorem~\ref{thm:abelian_SO_betti_bound},
\[
\sum_{i=0}^{n-1} b_i(X;\FF) \le 3^n,
\qquad\text{hence}\qquad
b_i(X;\FF)\le 3^n\ \text{ for all }i.
\]

If $\mathrm{char}(\FF)=0$, then $|Q|$ is invertible in $\FF$, and the transfer for the $Q$--action (as in
\cite[\S2.7]{AdemDavis}) gives an injection $H_*(X/Q;\FF)\hookrightarrow H_*(X;\FF)$, so
$\sum_i b_i(X/Q;\FF)\le 3^n$.

Assume $\mathrm{char}(\FF)=p>0$. Applying Theorem~\ref{thm:orbit_finite_bound} to the admissible $Q$--action on $X$ with
$d=n-1$ and $k\le 3^n$ yields, for each $0\le i\le n-1$,
\[
b_i(X/Q;\FF)\ \le\ |Q|^{\log_p(3n)}\,3^n\ \le\ |Q|^{\log_2(3n)}\,3^n,
\]

using $\log_p(\cdot)\le \log_2(\cdot)$ for $p\ge 2$. Summing over $i=0,\dots,n-1$ gives the stated bound.
\end{proof}

\section{The quasi-thick part of non-positively curved orbifolds}

This section explains how Theorem~\ref{thm:intro main topological theorem} is used to deduce
Theorem~\ref{thm:intro main theorem}.  The argument follows the general thick--thin strategy of
Gromov and Gelander, together with Samet's modification for orbifolds. The main idea is to divide the space into a thick part, which can be covered by a good cover of contractible balls, and a thin part, where the topology is controlled separately, mainly due to the works of Margulis. In the orbifold case, instead of separating the orbifold
into a thick part and a thin part, one enlarges the thick part by attaching the singular locus, obtaining
a \emph{quasi-thick} region whose topology can still be controlled by a good cover.
Then, a $\Gamma$--invariant Morse function built from the displacement function decomposes the topology of
$O$ into finitely many local parts near the critical points. The topology of each part can then be modelled by a normal bundle,
whose homology is controlled by a sphere bundle over a union of balls. 
Our new input compared to \cite{samet2013betti} is that the local contributions of singularities
can be bounded uniformly over \emph{arbitrary} fields, using Theorem~\ref{thm:intro main topological theorem}. For additional details, we refer the reader to \cite{samet2013betti}.

\subsection{Setting, displacement, and quasi-thick parts}\label{subsec:qt_setting}

Let $X$ be an $n$--dimensional Hadamard manifold with sectional curvatures normalized so that
$-1\le K\le 0$, and let $\Gamma\le \Isom(X)$ be a lattice.  Set $O:=X/\Gamma$ and let
$\pi:X\to O$ be the quotient map.

For $\gamma\in\Gamma$ we write
\[
d_\gamma(x):=d(x,\gamma x)
\qquad\text{and}\qquad
\Min(\gamma):=\{x\in X: d_\gamma(x)=\inf d_\gamma\}.
\]
Given $\varepsilon>0$ and $x\in X$, let
\[
\Gamma_\varepsilon(x):=\big\langle \{\gamma\in\Gamma:\ d_\gamma(x)\le \varepsilon\}\big\rangle.
\]
For $m\in\N$ define the $\Gamma$--invariant subset
\[
X_{\ge \varepsilon,m}:=\{x\in X:\ |\Gamma_\varepsilon(x)|\le m\},
\qquad
O_{\ge \varepsilon,m}:=X_{\ge \varepsilon,m}/\Gamma.
\]
(Thus points with ``many'' small-displacement elements are removed, but torsion points with bounded stabilizer
are allowed.)

We will measure the size of $O$ using the notion of \emph{essential volume} at a fixed scale:
for $\rho>0$ define $\operatorname{ess\text{-}vol}_\rho(O)$ to be the supremum of the cardinalities of
$2\rho$--separated subsets of the $\rho$--thick part of $O$ (equivalently, the maximal number of disjoint
embedded $\rho$--balls in the $\rho$--thick part).
By standard packing/comparison estimates (using the curvature normalization), for each fixed $\rho$ one has
\begin{equation}\label{eq:essvol_vs_vol}
\operatorname{ess\text{-}vol}_\rho(O)\ \le\ C(n,\rho)\,\Vol(O).
\end{equation}
We will therefore aim for bounds in terms of $\operatorname{ess\text{-}vol}_\rho(O)$.

\subsection{A good cover of the quasi-thick part}

The key geometric input is that the quasi-thick part admits a uniformly bounded-overlap good cover.

\begin{thm}[Samet {\cite[Thm.~4.2]{samet2013betti}}]\label{thm:samet_good_cover}
Fix $\varepsilon>0$ and $m\in\N$.  There exist constants $\rho>0$ and $r\in\N$, depending only on
$n,\varepsilon,m$, such that $O_{\ge \varepsilon,m}$ is contained in a finite collection $\mathcal B$ of open
metric balls in $O$ with:
\begin{enumerate}[label=(\arabic*)]
\item each $B\in\mathcal B$ has radius at most $\varepsilon/4$;
\item every nonempty finite intersection of balls in $\mathcal B$ is contractible;
\item $|\mathcal B|\le \operatorname{ess\text{-}vol}_\rho(O)$;
\item each $B\in\mathcal B$ meets at most $r$ other balls of the cover.
\end{enumerate}
Moreover, if $\operatorname{ess\text{-}vol}_\rho(O)<\infty$ then $O_{\ge \varepsilon,m}$ is compact.
\end{thm}

\subsection{Bundles over a cover}
Samet proves a homology bound for a fibration over $\bigcup\mathcal B$ with characteristic-$0$ coefficients
\cite[Prop.~4.11]{samet2013betti}.  The main point where characteristic $0$ is used is in controlling the
homology of quotients of certain fibres.  Our Theorem~\ref{thm:intro main topological theorem}
replaces this input and yields the same conclusion over an arbitrary field.

\begin{prop}\label{main fibration proposition}
Let $\mathcal B$ be a cover as in Theorem~\ref{thm:samet_good_cover}, let $U:=\bigcup_{B\in\mathcal B}B$, and
let $E\to U$ be a (topological) fibre bundle whose fibre $F$ is either a sphere or a real vector space.
Then there exists a function $h=h(r)$ such that for every field $\F$ and every $j$,
\[
\dim_{\F}H_j(E;\F)\ \le\ h(r)\cdot|\mathcal B|
\ \le\ h(r)\cdot\operatorname{ess\text{-}vol}_\rho(O),
\]
where $\rho$ is the scale from Theorem~\ref{thm:samet_good_cover}.
\end{prop}

\begin{proof}
Let $V$ be a nonempty intersection of balls from $\mathcal B$.  By construction in \cite{samet2013betti},
$V$ lifts to a contractible set $\widetilde V\subset X$ on which a finite subgroup
$\Gamma_{\widetilde V}\le \Gamma$ (of order $\le m$) acts, with $V\cong \widetilde V/\Gamma_{\widetilde V}$.
The restriction $E|_V$ is then a quotient of a trivial bundle $\widetilde V\times F$ by $\Gamma_{\widetilde V}$.
Because $\widetilde V$ admits a $\Gamma_{\widetilde V}$--equivariant contraction to a fixed point, $E|_V$
deformation retracts to $F/\Gamma_{\widetilde V}$.

If $F$ is a vector space, then $F/\Gamma_{\widetilde V}$ is contractible, hence $E|_V$ is contractible.
If $F$ is a sphere, then the action on the fibre is linear (coming from the action on the normal sphere at a
fixed point), so Theorem~\ref{thm:intro main topological theorem} gives a uniform bound (depending only on
$\dim F$) on the Betti numbers of $F/\Gamma_{\widetilde V}$ over \emph{any} field.

Thus $E$ is covered by the sets $E|_B$ ($B\in\mathcal B$), each meeting at most $r$ others, and every nonempty
finite intersection has homology uniformly bounded over $\F$.  A standard Mayer--Vietoris induction for bounded
degree covers (cf.\ \cite[Lem.~12.12]{BGS85}) yields
$\dim_\F H_j(E;\F)\le h(r)|\mathcal B|$ for a function $h$ depending only on $r$.
\end{proof}

\subsection{The Morse function and critical pairs}\label{subsec:morse_function}

We now recall the Morse-theoretic reduction from \cite[\S5]{samet2013betti} in a way that isolates the inputs
we will use later.

\smallskip\noindent
\textbf{Margulis constants.}
Let $\varepsilon_n>0$ and $m_n\in\N$ be the constants given by the Margulis lemma for $n$--manifolds: for every discrete $\Gamma<\Isom(X)$ and every $x\in X$,
$\Gamma_{\varepsilon_n}(x)$ contains a normal nilpotent subgroup of index $\le m_n$, and if
$\Gamma_{\varepsilon_n}(x)$ is finite then it is virtually abelian.
Fix $\varepsilon:=\varepsilon_n$ and $m:=m_n$.

We also use the following quantitative lemma ( \cite[Lem.~5.2]{samet2013betti}):

\begin{lem}\label{lem:samet_5_2}
There exist constants $\delta>0$ and $M_1\in\N$ (depending only on $n$) such that:
if $d(x,y)<2\varepsilon$ and $d_\gamma(x)\le \delta$, then $d_{\gamma^i}(y)\le \varepsilon$ for some $1\le i\le M_1$.
\end{lem}

\smallskip\noindent
\textbf{Stability.}
Set $M_2:=mM_1$.
For $s\in\N$, call $g\in SO(n)$ \emph{$s$--stable} if $C_{SO(n)}(g)=C_{SO(n)}(g^i)$ and
$\Fix(g)=\Fix(g^i)$ for $1\le i\le s$.  An elliptic isometry $\gamma\in\Isom(X)$ is \emph{$s$--stable} if the
linear action of $\gamma$ on $T_xX$ at a fixed point $x$ is $s$--stable.
By \cite[Prop.~2.5]{samet2013betti}, there exists a constant $J=J(n,M_2)$ such that for every isometry $\gamma$
there is some $1\le j\le J$ with $\gamma^j$ $M_2$--stable.

\smallskip\noindent
\textbf{The function $F$.}
Choose smooth functions
\[
g_k:\R_{\ge 0}\to \R_{\ge 0}\quad (k\ge 2),
\qquad
g_\infty:\R_{>0}\to \R_{\ge 0},
\]
such that:
(i) each is strictly decreasing on $(0,\varepsilon)$ and vanishes on $[\varepsilon,\infty)$;
(ii) $g_k(0)=k$ and $g_k(\delta)=1$ for $k\ge 2$;
(iii) $\lim_{t\to 0^+} g_\infty(t)=\infty$ and $g_\infty(\delta)=1$.
For $1\ne \gamma\in\Gamma$ let $o(\gamma)\in\N\cup\{\infty\}$ be its order and put $g_\gamma:=g_{o(\gamma)}$.

Define the set of ``relevant'' elements
\[
\Delta:=\left\{1\ne \gamma\in\Gamma:\ \inf d_\gamma<\delta,\ \text{and if $\gamma$ is elliptic then it is $M_2$--stable}\right\}
\]
and the function
\[
F(x):=\sum_{\gamma\in \Delta}\ \sum_{i=0}^{M_1}\ g_\gamma\!\bigl(d_{\gamma^i}(x)\bigr).
\]
Since $\Gamma$ is discrete, only finitely many summands are nonzero near each $x$, hence $F$ is smooth.
It is $\Gamma$--invariant, so it descends to a smooth function $f$ on $O=X/\Gamma$.

For each $x\in X$, let
\[
\Delta_x:=\left\{\gamma^i:\ \gamma\in\Delta,\ 0\le i\le M_1,\ d_{\gamma^i}(x)<\varepsilon\right\},
\]
i.e.\ the elements that actually contribute near $x$.

\smallskip\noindent
\textbf{Critical pairs.}
The following package of facts is proved in \cite[\S5]{samet2013betti} (Claims 1--7 there).  We record them in the
form we will use.

\begin{claim}\label{claim:proper}
The function $f$ is proper.  In particular, each sublevel set $f_{\le t}$ is compact.
\end{claim}

\begin{claim}\label{claim:critical_pairs}
A point $x\in X$ is a critical point of $F$ if and only if it lies in
\[
Y_x:=\bigcap_{\eta\in \Delta_x}\Min(\eta).
\]
Moreover, if $x$ is critical and we set
\[
C_x:=\{y\in Y_x:\ \Delta_y=\Delta_x\},
\]
then $C_x$ is exactly the critical set of $F$ lying in $Y_x$, and for any $y\in C_x$ one has $Y_y=Y_x$ and $C_y=C_x$.
We call $(Y_x,C_x)$ the \emph{critical pair} associated to $x$.
\end{claim}

Fix a critical pair $(Y,C)$.  Let
\[
C':=\{y\in Y:\ d(y,C)<\varepsilon/2\}.
\]

\begin{claim}\label{claim:local_invariance_and_injection}
There is a neighborhood of $C$ in $X$ on which small-displacement elements preserve $Y$; in particular the natural map
\[
C'/\Gamma_Y\ \longrightarrow\ X/\Gamma
\]
is injective, where $\Gamma_Y:=\{\gamma\in\Gamma:\ \gamma Y=Y\}$.
\end{claim}

\begin{claim}\label{claim:transverse_decrease}
Let $x\in C'$ and let $c:[0,1)\to X$ be a geodesic ray with $c(0)=x$ orthogonal to $Y$.
Then $t\mapsto F(c(t))$ is strictly decreasing for $t>0$ sufficiently small.
\end{claim}

We now introduce the constants used to describe neighbourhoods of $C$ inside $Y$.  Set
\[
\delta_2:=\frac{\delta}{2mJ+2},
\qquad
M_3:=m\,(J!)^n.
\]

\begin{claim}\label{claim:structure_of_Y}
Let $(Y,C)$ be a critical pair.
If $\Delta_C$ contains a hyperbolic element, then $Y=C$ and $Y$ is a geodesic line.
Otherwise, $Y$ is an $m$--stable singular submanifold and there exists a $\Gamma_Y$--invariant neighborhood
\[
C''\subset C'
\]
of $C$ in $Y$ such that $C''/\Gamma_Y$ is contained in the $(\delta_2,M_3)$--quasi-thick part of $Y/\Gamma_Y$.
\end{claim}

\smallskip\noindent
\textbf{Finiteness of critical values.}
Claim~\ref{claim:structure_of_Y} reduces the critical pairs to two types.
When $K<0$, the hyperbolic (geodesic) type is quantitatively controlled by \cite[Thm.~3.5]{samet2013betti}.
The remaining (purely elliptic) type is controlled by \cite[Thm.~3.2]{samet2013betti} together with the good cover
of the quasi-thick part (Theorem~\ref{thm:samet_good_cover}); in particular there are only finitely many critical pairs
up to $\Gamma$--conjugacy, hence $f$ has finitely many critical values.  We will also use that by compactness considerations, the critical sets project to
compact subsets of $O$.

Let
\[
0=c_1<c_2<\cdots<c_k
\]
be the distinct critical values of $f$.  Choose $\mu>0$ so small that each interval $[c_i-\mu,c_i+\mu]$ contains no
critical values other than $c_i$, and these intervals are pairwise disjoint.
Since $f$ is proper (Claim~\ref{claim:proper}), standard Morse-theoretic deformation arguments (cf.\ \cite{milnor1963morse})
show that the gradient flow yields deformation retractions
\[
O\ \simeq\ f_{\le c_k+\mu},
\qquad
f_{\le c_{i+1}-\mu}\ \simeq\ f_{\le c_i+\mu}.
\]
Using subadditivity of ranks for triples gives, for every $j$ and every field $\F$,
\begin{equation}\label{eq:rank_inequality_rewrite}
\dim_{\F}H_j(O;\F)\ \le\ \sum_{i=1}^k \dim_{\F}H_j\bigl(f_{\le c_i+\mu},\,f_{\le c_i-\mu};\F\bigr).
\end{equation}
Thus, it remains to bound the relative groups on the right.

Fix one critical value $c$ and fix a corresponding critical pair $(Y,C)$ (up to $\Gamma$--conjugacy).
Let
\[
V:=f_{\le c+\mu}\cap \pi(Y)\subset O.
\]
For $\mu$ sufficiently small we may assume $V\subset \pi(C'')$ (where $C''$ is from Claim~\ref{claim:structure_of_Y}).
Consider the triple
\[
f_{\le c-\mu}\ \subset\ (f_{\le c+\mu}\setminus V)\ \subset\ f_{\le c+\mu}.
\]
Claim~\ref{claim:transverse_decrease} implies that, near $\pi(C'')$, the negative gradient flow does not move points
towards $\pi(Y)$.  Since $f$ has no critical points in $f_{\le c+\mu}\setminus V$, one obtains a deformation retraction
\[
f_{\le c+\mu}\setminus V\ \simeq\ f_{\le c-\mu},
\]
hence
$H_*\bigl(f_{\le c+\mu}\setminus V,\,f_{\le c-\mu};\F\bigr)=0$.
The long exact sequence of the triple yields
\[
H_*\bigl(f_{\le c+\mu},\,f_{\le c-\mu};\F\bigr)\ \cong\ H_*\bigl(f_{\le c+\mu},\,f_{\le c+\mu}\setminus V;\F\bigr).
\]
By excision, the latter depends only on a small neighborhood of $V$ in $f_{\le c+\mu}$.
Taking an orbifold tubular neighborhood identifies that neighborhood with the (orbifold) normal bundle $N(V)$, and we get
\begin{equation}\label{eq:tubular_reduction_rewrite}
H_j\bigl(f_{\le c+\mu},\,f_{\le c-\mu};\F\bigr)\ \cong\
H_j\bigl(N(V),\,N(V)\setminus N_0(V);\F\bigr),
\end{equation}
where $N_0(V)$ is the zero section.

\subsection{Bounding the tubular contribution}

We now bound the right-hand side of \eqref{eq:tubular_reduction_rewrite} in the two cases of
Claim~\ref{claim:structure_of_Y}.

\medskip\noindent
\textbf{Case 1: there is a hyperbolic element.}
Then $Y=C$ is a geodesic line, so $V$ is $1$--dimensional (either contractible or a circle in the quotient).
In particular, $N(V)$ is homotopy equivalent to $V$, and $N(V)\setminus N_0(V)$ is homotopy equivalent to a
sphere bundle over $V$.  It follows that for all $j$,
\[
\dim_{\F}H_j\bigl(N(V),\,N(V)\setminus N_0(V);\F\bigr)\ \le\ 3.
\]

\medskip\noindent
\textbf{Case 2: the critical pair is purely elliptic.}
Then $V\subset \pi(C'')$ lies in the $(\delta_2,M_3)$--quasi-thick part of $Y/\Gamma_Y$.
Apply Theorem~\ref{thm:samet_good_cover} \emph{inside $Y/\Gamma_Y$} with parameters $(\delta_2,M_3)$.
This gives constants $\rho_3$ and $r_3$ (depending only on $n$) and a cover of a neighborhood of $V$ in $Y/\Gamma_Y$ by
balls of radius $\le \delta_2/4$ with contractible intersections, bounded overlap $r_3$, and
cardinality bounded by $\operatorname{ess\text{-}vol}_{\rho_3}(Y/\Gamma_Y)$.

Let $U_Y$ be the union of those balls and consider the sphere bundle model for
$N(U_Y)\setminus N_0(U_Y)\to U_Y$.  Using Proposition~\ref{main fibration proposition} for this sphere bundle (and for
the corresponding bundle over $U_Y$ itself) gives a bound, uniform over all fields $\F$,
\begin{align*}
\dim_{\F}H_j\bigl(N(U_Y);\F\bigr)
&\le h(r_3)\,\operatorname{ess\text{-}vol}_{\rho_3}(Y/\Gamma_Y),\\
\dim_{\F}H_j\bigl(N(U_Y)\setminus N_0(U_Y);\F\bigr)
&\le h(r_3)\,\operatorname{ess\text{-}vol}_{\rho_3}(Y/\Gamma_Y).
\end{align*}
To compare $V$ with $U_Y$, one uses a choice of the functions $g_k$:
by taking $g_k(\varepsilon/2)$ sufficiently small, one ensures that in a fixed neighbourhood of $C$ inside $Y$,
the values of $F$ remain within a controlled window above the minimum on $C$ (cf.\ \cite[\S5,(7)]{samet2013betti}).
In particular, one can arrange that $V$ is a deformation retract of a slightly larger sublevel region that contains $U_Y$.
This yields that the homology of the pairs for $V$ is bounded by that for $U_Y$, and hence
\begin{equation}\label{eq:tubular_bound_rewrite}
\dim_{\F}H_j\bigl(N(V),\,N(V)\setminus N_0(V);\F\bigr)
\ \le\ 2h(r_3)\,\operatorname{ess\text{-}vol}_{\rho_3}(Y/\Gamma_Y).
\end{equation}

\subsection{Summing over critical pairs.}

Insert the bounds from Case~1 and \eqref{eq:tubular_bound_rewrite} into
\eqref{eq:rank_inequality_rewrite}.  The number of hyperbolic-type critical pairs is controlled (when $K<0$)
by \cite[Thm.~3.5]{samet2013betti}, giving a contribution bounded by a constant times
$\operatorname{ess\text{-}vol}_{\varepsilon}(O)$.

For elliptic-type pairs, apply \cite[Thm.~3.2]{samet2013betti} to a maximal set of non-conjugate $m$--stable singular
submanifolds $Y$, with $\varepsilon_1:=\rho_3$.  This yields a constant $\rho_4=\rho_4(n,\rho_3)$ such that
\[
\sum_Y \operatorname{ess\text{-}vol}_{\rho_3}(Y/\Gamma_Y)\ \le\ \operatorname{ess\text{-}vol}_{\rho_4}(O).
\]
Therefore, summing \eqref{eq:tubular_bound_rewrite} over the elliptic-type critical pairs bounds their total contribution
by a constant times $\operatorname{ess\text{-}vol}_{\rho_4}(O)$.

Combining both types, we obtain for each $j$:
\[
\dim_\F H_j(O;\F)\ \le\ C(n)\,\operatorname{ess\text{-}vol}_{\rho'}(O)
\]
for some $\rho'=\rho'(n)>0$ and constant $C(n)$ depending only on the dimension.
Finally, \eqref{eq:essvol_vs_vol} converts this into the desired linear bound in the volume, completing the proof of
Theorem~\ref{thm:intro main theorem}.

\bibliography{mybibliography}
\bibliographystyle{alpha}

\bigskip
  \footnotesize

  G.~Kapon, \textsc{Department of Mathematics, Weizmann Institute of Science, 234 Herzl Street, Rehovot 76100, Israel.}\par\nopagebreak

  \medskip

  R.~Slutsky, \textsc{Mathematical Institute,
University of Oxford,
Andrew Wiles Building,
Oxford, OX2 6GG, United Kingdom.}\par\nopagebreak
  \textit{E-mail address:} \texttt{slutsky@maths.ox.ac.uk}\par\nopagebreak

\end{document}